\RequirePackage{fix-cm}
\documentclass{svjour3}                     
\smartqed  
\usepackage[parfill]{parskip}    
\usepackage{amsmath, amssymb, latexsym, enumerate, mathrsfs}
\usepackage{graphicx}
\usepackage{hyperref}
\usepackage{graphicx}
\usepackage{array}
\usepackage{times}
\usepackage{moreverb}
\usepackage{marvosym}
\DeclareMathOperator{\rank}{rank}

\def\etal{\mbox{\em et al.}}
\usepackage{latexsym}
\begin{document}

\title{On quantification of systemic redundancy in reliable systems
}

\titlerunning{On quantification of systemic redundancy in reliable systems}        

\author{Getachew K. Befekadu}

\author{Getachew K. Befekadu \and Panos~J.~Antsaklis
}

\institute{G. K. Befekadu~ ({\large\Letter}\negthinspace) \at
	          Department of Electrical Engineering, University of Notre Dame, Notre Dame, IN 46556, USA. \\
	          \email{gbefekadu1@nd.edu}           
	          \and
	          P. J. Antsaklis \at
	          Department of Electrical Engineering, University of Notre Dame, Notre Dame, IN 46556, USA. \\
	          \email{antsaklis.1@nd.edu}           
}

\date{February 15, 2015}

\maketitle

\begin{abstract}
In this paper, we consider the problem of quantifying systemic redundancy in reliable systems having multiple controllers with overlapping functionality. In particular, we consider a multi-channel system with multi-controller configurations -- where controllers are required to respond optimally, in the sense of best-response correspondence, a {\it reliable-by-design} requirement, to non-faulty controllers so as to ensure or maintain some system properties. Here we introduce a mathematical framework, based on the notion of relative entropy of probability measures associated with steady-state solutions of Fokker-Planck equations for a family of stochastically perturbed multi-channel systems, that provides useful information towards a systemic assessment of redundancy in the system.  
\end{abstract}

\section{Introduction} \label{S1}
The notion of redundancy, which promotes robustness by ``backing-up" important functions of systems, together with its systemic quantification, has long been recognized as an essential design philosophy by researchers in different fields of studies (to mention a few, e.g., see \cite{KraP02}, \cite{Tau92} and \cite{Wag05} for related discussions in biological systems; see also \cite{CarD99}, \cite{KanB11}, \cite{SieS98} and \cite{RanLOT11} for related discussions in engineering systems). In this paper, we consider the problem of quantifying systemic redundancy in reliable systems having multiple controllers with overlapping functionality. To be more specific, we consider a multi-channel system with multi-controller configurations -- where controllers are required to respond optimally, a {\it reliable-by-design} requirement, to non-faulty controllers so as to maintain the stability of the system when there is a single failure in any of the control channels. 

Here we introduce a mathematical framework, based on the notion of relative entropy of probability measures associated with steady-state solutions of the Fokker-Planck equations for a family of stochastically perturbed multi-channel systems, that provides useful information towards a systemic assessment of redundancy in the system. For such steady-state solutions of the Fokker-Planck equations, we establish a quantifiable redundancy measure by using the difference between the average relative entropy of the steady-state probability measures, with respect to any single controller failure in the system, and the entropy of the steady-state probability measure under a nominal operating condition, i.e., without any single controller failure in the system. Moreover, we determine the asymptotic behavior of the systemic redundancy measure in the multi-channel system as the random perturbation decreases to zero, where such an asymptotic result can be related to the solutions of controlled Liouville equations for the underlying original unperturbed multi-channel system.

It is worth mentioning that some interesting studies on systemic measures, based on information theory, have been reported in literature (e.g., see \cite{ToSE99}, \cite{LiDHKY12} or \cite{EdekG01} in the context of complexity, degeneracy and redundancy measures in biological systems; see \cite{Kit04} or \cite{Kit07} in the context of robustness in biological systems; and see also \cite{Bru83} or \cite{Cru12} for related discussion on the complexity measure for trajectories in dynamical systems). Moreover, such studies have also provided some useful information in characterizing or understanding the systemic measures (based on mutual information between appropriately partitioned input and output spaces) in systems with multiple subsystems/modules having overlapping functionality. Note that the rationale behind our framework follows in some sense the settings of these papers. However, to our knowledge, this problem has not been addressed in the context of reliable systems with multi-controller configurations having ``overlapping or backing-up" functionality, and it is important because it provides a framework for quantifying or gauging systemic redundancy measure in multi-channel systems, for example, when there is a single channel failure in the system.

This paper is organized as follows. In Section~\ref{S2}, we present some preliminary results that are useful for our main results. Section~\ref{S3} presents our main results, where we introduce a mathematical framework that provides useful information towards a systemic assessment of redundancy in reliable systems. This section also contains a result on the asymptotic behavior of the systemic redundancy measure in the multi-channel system as the random perturbation decreases to zero.

\section{Preliminary results} \label{S2}

Consider the following continuous-time multi-channel system  
\begin{align} 
 \dot{x}(t) &= A x(t) + \sum\nolimits_{i=1}^N B_i u_i(t), \quad x(0)=x_0,  \label{Eq1}
\end{align}
where $x \in \mathbb{R}^{d}$ is the state, $u_i \in \mathbb{R}^{r_i}$ is the control input to the $i$th-channel.

Let $S$ be a compact manifold in $\mathbb{R}^{d}$ and let $\rho_0(x) \triangleq \rho(0, x) > 0$, with $\int_{S} \rho_0(x)dx=1$, be an initial density function. Further, let $\psi$ be a smooth function $\psi \colon S \rightarrow \mathbb{R}^{+}$ having compact support, then the expected value of $\psi$ at some future time $t > 0$ is given by
\begin{align}
E \bigl\{ \psi(x) \bigr\} = \int_{S} \psi(x) \rho(t, x) dx. \label{Eq2} 
\end{align}
Moreover, if we take the time derivative of the above equation and make use of integration by parts, then we will have
\begin{align}
\int_{S} \psi(x) \frac{\partial \rho(t, x)}{\partial t} dx = - \int_{S} \psi(x) \Bigl \langle \frac{\partial }{\partial x}, \Bigl (A x + \sum\nolimits_{i=1}^N B_i u_i\Bigr)\rho(t, x) \Bigr \rangle dx. \label{Eq3} 
\end{align}
Since $\psi(x)$ is an arbitrary function, we can rewrite the above equation as a first-order partial differential equation (which is also known as the Liouville equation)
\begin{align}
 \frac{\partial \rho(t, x)}{\partial t} = -  \Bigl \langle \frac{\partial }{\partial x}, \Bigl (A x + \sum\nolimits_{i=1}^N B_i u_i\Bigr)\rho(t, x) \Bigr \rangle.  \label{Eq4} 
\end{align}
\begin{remark} \label{R1}
Here we remark that the above first-order partial differential equation describes how the density function $\rho(t, x)$ evolves in time (i.e., Equation~\eqref{Eq4} describes an evolution equation on $L_1\bigl(\mathbb{R}^{d}\bigr)$ under a flow defined by the deterministic system of \eqref{Eq1}). Furthermore, it is easy to verify that
\begin{align*}
\rho(t, x) > 0, \quad t \ge 0,
\end{align*}
and further it satisfies
\begin{align*}
\frac{d}{d t} \int_{S} \rho(t, x) dx = 0.
\end{align*}
\end{remark}

Notice that the partial derivative with respect to $x$ in \eqref{Eq4} depends on whether the input controls $u_i$ are expressed as open-loop functions (i.e., $u_i=u_i(t)$ for $i=1,2, \ldots, N$) or as closed-loop functions (i.e., $u_i=u_i(t,x)$ for $i=1,2, \ldots, N$); and, as a result of this, the solution for $\rho(t, x)$ depends on the type of input controls used in the system.

\begin{remark} \label{R2}
Recently, using a class of variational problems, the author in \cite{Broc07} has considered the Liouville equations that involve control terms. Moreover, such a formulation is useful for relating the behavior of the solutions of the Liouville equation to that of the behavior of the underlying differential equation of the system.
\end{remark}

In what follows, we recall some known results that will be used for our main results (e.g., see \cite{Csi67} or \cite{Les14}).

\begin{definition}\label{D1}
Let $\mu$ be a probability measure on $\mathbb{R}^d$ with respect to the density function $\rho(t, x)$ which satisfies the Liouville equation in \eqref{Eq4} starting from an initial density function $\rho_0(x_0)$. Then, the entropy of $\mu$ (with respect to Lebesgue measure) is defined by
\begin{align}
 H\bigl(\mu\bigr) = - \int_{S} \rho(t, x) \log_2 \rho(t, x) dx, \quad t \ge 0. \label{Eq5}
\end{align}
\end{definition}

\begin{remark} \label{R3}
In general, we have the following inequality
\begin{align*}
- \int_{S}  \rho_1(t, x) \log_2 \rho_1(t, x) dx  \le - \int_{S} \rho_1(t, x) \log_2 \rho_2(t, x) dx,
\end{align*}
for any two density functions $\rho_1(t, x)$ and $\rho_2(t, x)$ that satisfy \eqref{Eq4} starting from $\rho_{1,0}(x_0)$ and $\rho_{2,0}(x_0)$, respectively.
\end{remark}

\begin{definition}\label{D2}
Let $\mu_1$ and $\mu_2$ be two probability measures on $\mathbb{R}^d$ with respect to the density functions $\rho_1(t, x)$ and $\rho_2(t, x)$ that satisfy the Liouville equation in \eqref{Eq4} starting from initial density functions $\rho_{1,0}(x_0)$ and $\rho_{2,0}(x_0)$, respectively. Then, the relative entropy of $\mu_2$ with respect to $\mu_1$ is defined by
\begin{align}
 D\bigl(\mu_2 \,\Vert\, \mu_1\bigr) & = \int_{S} \rho_2(t, x) \log_2 \left (\frac{\rho_2(t, x)}{\rho_1(t, x)}\right) dx\notag \\
                                                       &= \int_{S} \Bigl(\rho_2(t, x) \log_2 \rho_2(t, x) - \rho_2(t, x) \log_2 \rho_1(t, x) \Bigr) dx, \quad t \ge 0.  \label{Eq6}
\end{align}
\end{definition}

\begin{remark} \label{R4}
Note that the relative entropy (also called the Kullback-Leibler distance) $D\bigl(\mu_2 \,\Vert\, \mu_1\bigr)$, which measures the deviation of $\mu_2$ with respect to the probability measure $\mu_1$, is nonnegative, i.e., $D \bigl(\mu_2 \,\Vert\, \mu_1\bigr) \ge 0$, and $D\bigl(\mu_2 \,\Vert\, \mu_1\bigr) = 0$ if and only if $\mu_2 = \mu_1$.
\end{remark}

Next, consider the following stochastically perturbed multi-channel system

\begin{align}
  dx^{\epsilon}(t) = A x^{\epsilon}(t) dt + \sum\nolimits_{i=1}^N B_i u_i(t) dt + \epsilon \sigma(x^{\epsilon}(t)) dW(t), \quad x^{\epsilon}(0)=x_0, \label{Eq7}
\end{align}
where
\begin{itemize}
\item[-] $x^{\epsilon}(\cdot)$ is an $\mathbb{R}^{d}$-valued diffusion process, $\epsilon > 0$ is a small parameter that represents the level of random perturbation in the system, 
\item[-] $\sigma \colon \mathbb{R}^{d} \rightarrow \mathbb{R}^{d \times m}$ is Lipschitz continuous with the least eigenvalue of $\sigma(\cdot)\sigma^T(\cdot)$ uniformly bounded away from zero, i.e., 
\begin{align*}
 \sigma(x)\sigma^T(x)  \ge \kappa I_{d \times d} , \quad \forall x \in \mathbb{R}^{d},
\end{align*}
for some $\kappa > 0$,
\item[-] $W(\cdot)$ (with $W(0)=0$) is an $m$-dimensional standard Wiener process, and
\item[-]  $u_i(\cdot)$ is a $U_i$-valued measurable control process to the $i$th-channel (i.e., an admissible control from the set $U_i \subset \mathbb{R}^{r_i}$) such that, for all $t > s$, $W(t) - W(s)$ is independent of $u_i(\nu)$ for $\nu > s$.
\end{itemize}

Note that the time evolution of the density function $\rho^{\epsilon}(t, x)$ associated with \eqref{Eq7} satisfies the following second-order partial differential equation (i.e., the Fokker-Planck equation)  
\begin{align}
 \frac{\partial \rho^{\epsilon}(t, x)}{\partial t}  =  - \Bigl \langle \frac{\partial} {\partial x},\, \Bigl( A x + \sum\nolimits_{i=1}^N B_i u_i\Bigr) \rho^{\epsilon}(t, x) \Bigr \rangle + & \frac{\epsilon^2}{2} \Bigl \langle \frac{\partial^2} {\partial x^2},\, \sigma(x)\sigma^T(x) \rho^{\epsilon}(t, x) \Bigr \rangle, \notag \\ 
 & \quad\quad \rho^{\epsilon}(0, x) \,\, \text{is given}. \label{Eq8}
\end{align}

In this paper, among all solutions of the above Fokker-Planck equation, we will only consider the steady-state solution that satisfies the following stationary Fokker-Planck equation\footnote{For example, see \cite{Kif74} or \cite{VenFre70} for additional discussion on the limiting behavior of invariant measures for systems with small random perturbation.}
\begin{align}
0  =  - \Bigl \langle \frac{\partial} {\partial x},\, \Bigl( A x + \sum\nolimits_{i=1}^N B_i u_i\Bigr) \rho_{\ast}^{\epsilon}(x) \Bigr \rangle + & \frac{\epsilon^2}{2} \Bigl \langle \frac{\partial^2} {\partial x^2},\, \sigma(x)\sigma^T(x) \rho_{\ast}^{\epsilon}(x) \Bigr \rangle, \notag \\
 \rho_{\ast}^{\epsilon}(x) > 0, & \quad \int_{\mathbb{R}^d} \rho_{\ast}^{\epsilon}(x) dx = 1. \label{Eq9}
\end{align}

In the following section, i.e., Section~\ref{S3}, such a steady-state solution, together with the solution of the Liouville equation, will allow us to provide useful information towards a systemic assessment of redundancy in the reliable multi-channel systems with small random perturbation.

\section{Main results} \label{S3}
In what follows, we consider a particular class of stabilizing state-feedbacks that satisfies\footnote{$\operatorname{Sp}(A)$ denotes the spectrum of a matrix $A \in \mathbb{R}^{d \times d}$, i.e., $\operatorname{Sp}(A) = \bigl\{s \in \mathbb{C}\,\bigl\lvert \,\rank(A - sI) < d \bigr\}$.}
\begin{align}
  \mathcal{K} \subseteq \Biggl\{\bigl(K_1, K_2, \ldots, K_N\bigr) &  \in \prod\nolimits_{i=1}^N \mathbb{R}^{r_i \times d} \biggm \lvert \operatorname{Sp}\Bigl(A + \sum\nolimits_{i=1}^N B_{i} K_i\Bigr) \subset \mathbb{C}^{-} \,\, \& \notag \\
    & \operatorname{Sp}\Bigl(A + \sum\nolimits_{i \neq j}^N B_{i} K_i\Bigr) \subset \mathbb{C}^{-},\,\, j = 1,2, \ldots, N \Biggr\}. \label{Eq10}
\end{align}
\begin{remark} \label{R5}
We remark that the above class of state-feedbacks is useful for maintaining the stability of the closed-loop system both when all of the controllers work together, i.e., $\bigl(A + \sum\nolimits_{i=1}^N B_{i} K_i\bigr)$, as well as when there is a single-channel controller failure in the system, i.e., $\bigl(A + \sum\nolimits_{i \neq j}^N B_{i} K_i\bigr)$ for $j=1,2, \ldots, N$. Moreover, such a class of state-feedbacks falls within the redundant/passive fault tolerant controller configurations with overlapping functionality, i.e., a reliable-by-design requirement in the system (e.g., see \cite{BefGA14} or \cite{FujBe09}).   
\end{remark}

Note that, for such a class of state-feedbacks, the density functions $\rho^{(j)}(t, x)$ for $j = 0,1, \ldots, N$, satisfy the following family of Liouville equations
\begin{align}
 \frac{\partial \rho^{(0)}(t, x)}{\partial t} = -  \Bigl \langle \frac{\partial }{\partial x}, \Bigl (Ax + \sum\nolimits_{i=1}^N B_i K_i x \Bigr) \rho^{(0)}(t, x) \Bigr \rangle \label{Eq11}
\end{align}
and 
\begin{align}
 \frac{\partial \rho^{(j)}(t, x)}{\partial t} = -  \Bigl \langle \frac{\partial }{\partial x}, \Bigl (Ax + \sum\nolimits_{i \neq j}^N B_i K_i x \Bigl) \rho^{(j)}(t, x) \Bigr \rangle, \,\, j=1,2, \ldots, N, \label{Eq12}
\end{align}
 starting from an initial density function $\rho_0(x_0)$. Moreover, using the fact that the solution of 
\begin{align}
  \dot{x}(t) = \Bigl (A + \sum\nolimits_{i=1}^N B_i K_i \Bigr)x(t), \quad x(0)=x_0, \label{Eq13}
\end{align}
can be written as
\begin{align}
  x(t) = \exp(A t) x(0)+ \int_{0}^t \exp(A (t - \lambda)) \sum\nolimits_{i=1}^N B_i K_i x(\lambda) d\lambda. \label{Eq14}
\end{align}
Then, the density function $\rho^{(0)}(t, x)$ (when $j=0$) corresponding to the Liouville equation in \eqref{Eq11}, with an initial density $\rho_0(x_0)$, is given by
\begin{align}
  \rho^{(0)}(t, x) = \frac{1}{\exp(\operatorname{tr}A t)} \rho_0 \biggl(\exp(- A t)\biggl(x(t) &- \int_{0}^t \exp(A (t - \lambda)) \sum\nolimits_{i=1}^N B_i K_i x(\lambda) d\lambda \biggr ) \biggr). \label{Eq15}
\end{align}
Similarly, the density functions $\rho^{(j)}(t, x)$ (when $j=1,2, \ldots, N$) corresponding to the Liouville equations in \eqref{Eq12} are given by
\begin{align}
  \rho^{(j)}(t, x) = \frac{1}{\exp(\operatorname{tr}A t)} \rho_0 \biggl(\exp(- A t)\biggl(x(t) &- \int_{0}^t \exp(A (t - \lambda))\sum\nolimits_{i \neq j}^N B_i K_i x(\lambda) d\lambda \biggr ) \biggr). \label{Eq16}
\end{align}

On the other hand, for a given random perturbation $(\sigma, \epsilon)$, the steady-state density functions $\rho^{(\epsilon, j)}(x) $ for $j = 0,1, \ldots, N$, satisfy the following family of stationary Fokker-Planck equations
\begin{align}
0  =  - \Bigl \langle \frac{\partial} {\partial x},\, \Bigl ( A x + \sum\nolimits_{i=1}^N B_i K_i x \Bigr) \rho^{(\epsilon, 0)}(x) \Bigr \rangle + & \frac{\epsilon^2}{2} \Bigl \langle \frac{\partial^2} {\partial x^2},\, \sigma(x)\sigma^T(x) \rho^{(\epsilon, 0)}(x) \Bigr \rangle, \notag \\
 \rho^{(\epsilon, 0)}(x) > 0, \quad &  \int_{\mathbb{R}^d} \rho^{(\epsilon, 0)}(x) dx = 1, \label{Eq17}
\end{align}
and
\begin{align}
0  =  - \Bigl \langle \frac{\partial} {\partial x},\, \Bigl ( A x + \sum\nolimits_{i \neq j}^N B_i K_i x \Bigr ) \rho^{(\epsilon, j)}(x) \Bigr \rangle + & \frac{\epsilon^2}{2} \Bigl \langle \frac{\partial^2} {\partial x^2},\, \sigma(x)\sigma^T(x) \rho^{(\epsilon, j)}(x) \Bigr \rangle, \notag \\
 \rho^{(\epsilon, j)}(x) > 0, \quad \int_{\mathbb{R}^d} \rho^{(\epsilon, j)}(x) dx = 1,& \quad j = 1, 2, \ldots, N. \label{Eq18}
\end{align}

\begin{remark} \label{R6}
Note that $\sigma(x)$ is Lipschitz continuous, with the least eigenvalue of $\sigma(\cdot)\sigma^T(\cdot)$ uniformly bounded away from zero. Further, if it is twice differentiable on $\mathbb{R}^d$, the uniqueness for smooth steady-state solutions for \eqref{Eq17} and \eqref{Eq18} depends on the behavior of the original unperturbed multi-channel system as well as on the type of input controls used in the system.
\end{remark}

The following proposition provides a condition on the uniqueness of the solutions for the stationary Fokker-Planck equations. 
\begin{proposition} \label{P1}
Suppose that there exists at least one $N$-tuple stabilizing state-feedbacks that satisfies the conditions in \eqref{Eq10}. Then, there exist unique smooth density functions $\rho^{(\epsilon, j)}(x)$ for $j = 0,1, \ldots, N$, with respect to $(\sigma, \epsilon)$, corresponding to the stationary Fokker-Planck equations in \eqref{Eq17} and \eqref{Eq18}. 

Furthermore, the relative entropy of $\mu^{(\epsilon, i)}$ for $i = 1, 2, \ldots, N$, with respect to $\mu^{(\epsilon, 0)}$ (i.e., probability measures associated with the density functions $\rho^{(\epsilon, i)}(x)$ and $\rho^{(\epsilon, 0)}(x)$, respectively) is finite, i.e.,
\begin{align}
  D\bigl(\mu^{(\epsilon, i)} \,\Vert\, \mu^{(\epsilon, 0)}\bigr) < +\infty. \label{Eq19}
\end{align}
\end{proposition}

\begin{proof}
Suppose that $\sigma(x)$ is twice differentiable on $\mathbb{R}^d$. Let the $N$-tuple of state-feedbacks $\bigl(K_1, K_2, \ldots, K_N \bigr)$ satisfy the conditions in \eqref{Eq10}. Then, the origin is a stable equilibrium point for the original unperturbed multi-channel system in \eqref{Eq1} (with respect to this particular set of state-feedbacks). Moreover, for sufficiently small $\epsilon > 0$, there exists a Lyapunov function $V(x) > 0$, with $\lim_{\vert x \vert \rightarrow \infty} V(x) = +\infty$, such that\footnote{Note that the assumption of a common Lyapunov function $V(x)$ is not necessary in \eqref{Eq20} and \eqref{Eq21}.}
\begin{align}
 \Bigl \langle \Bigl ( A x + \sum\nolimits_{i=1}^N B_i K_i x\Bigl) \frac{\partial} {\partial x},\, V(x)\Bigr \rangle + \frac{\epsilon^2}{2} \Bigl \langle \sigma(x)\sigma^T(x) \frac{\partial^2} {\partial x^2},\, V(x)\Bigr \rangle < - \eta, \label{Eq20}
\end{align}
and
\begin{align}
 \Bigl \langle \Bigl ( A x + \sum\nolimits_{i \neq j}^N B_i K_i x \Bigr ) \frac{\partial} {\partial x},\, V(x)\Bigr \rangle + \frac{\epsilon^2}{2} \Bigl \langle \sigma(x)\sigma^T(x) \frac{\partial^2} {\partial x^2},\, V(x)\Bigr \rangle < - \eta, \notag \\
 j =1,2, \dots, N, \label{Eq21}
\end{align}
for all $x \in \mathbb{R}^d \backslash \{0\}$ and for some constant $\eta > 0$. Then, from Theorem~2.1 and Theorem~5.7 in \cite{BogKrB09}, the stationary measures $\mu^{(\epsilon, j)}$ for $j =0, 1, \dots, N$, of the Fokker-Planck equations in \eqref{Eq17} and \eqref{Eq18} uniquely admit smooth density functions $\rho^{(\epsilon, j)} \in C^{\infty}(\mathbb{R}^d)$ for $j =0, 1, \dots, N$, with $\rho^{(\epsilon, j)} > 0$ on $\mathbb{R}^d$, i.e.,
\begin{align*}
\mu^{(\epsilon, j)}(dx) = \rho^{(\epsilon, j)}(x) dx, \quad j =0, 1, \dots, N.
\end{align*}
Furthermore, the measure $\mu^{(\epsilon, i)}$ for $i \in \{1, 2, \ldots, N\}$, is absolutely continuous with respect to $\mu^{(\epsilon, 0)}$ (i.e., $\mu^{(\epsilon, i)} \ll \mu^{(\epsilon, 0)}$, $i = 1, 2, \ldots, N$) and, as a result, the relative entropy of $\mu^{(\epsilon, i)}(x)$ with respect to $\mu^{(\epsilon, 0)}(x)$ satisfies the following
\begin{align*}
  D\bigl(\mu^{(\epsilon, i)} \,\Vert\, \mu^{(\epsilon, 0)}\bigr) < +\infty, \quad i = 1, 2, \ldots, N.
 \end{align*}
This completes the proof.
\end{proof}

Next, let us define the systemic redundancy $r_{(\sigma, \epsilon)} \in \mathbb{R}$ (with respect to the random perturbation $(\sigma, \epsilon)$) as follows
\begin{align}
  r_{(\sigma, \epsilon)} = \frac{1}{2N} \sum\nolimits_{i=1}^N D\bigl(\mu^{(\epsilon, i)} \,\Vert\, \mu^{(\epsilon, 0)}\bigr) - H\bigl(\mu^{(\epsilon, 0)}\bigr), \label{Eq22}
\end{align}
where $\bigl(1/2N\bigr) \sum\nolimits_{i=1}^N D\bigl(\mu^{(\epsilon, i)} \,\Vert\, \mu^{(\epsilon, 0)}\bigr)$ represents the average relative entropy of probability measures with respect to any single failure in the control channels.

\begin{remark} \label{R7}
Here we remark that $r_{(\sigma, \epsilon)}$ provides useful information in characterizing the systemic redundancy (i.e., with respect to the state-space and partitioned input spaces) in the system with multi-controller configurations having overlapping functionality, i.e., a requirement to maintain the stability of the system when there is a single failure in any of the control channels.
\end{remark}

Then, we have the following result on the asymptotic property of the systemic redundancy measure $r_{(\sigma, \epsilon)}$ as the random perturbation decreases to zero (i.e., as $\epsilon \rightarrow 0$).
\begin{corollary}\label{C1}
Suppose that Proposition~\ref{P1} holds, then the redundancy $r_{(\sigma, \epsilon)}$ satisfies the following asymptotic property
\begin{align}
  r_{(\sigma, \epsilon)} \rightarrow r_{\infty} \quad \text{as} \quad \epsilon \rightarrow 0, \label{Eq23}
\end{align}
where
\begin{align}
  r_{\infty} &=  \lim_{ t \rightarrow \infty} \Biggl[\underbrace{\frac{1}{2N} \sum\nolimits_{i=1}^N D\bigl(\mu^{(i)} \,\Vert\, \mu^{(0)}\bigr) - H\bigl(\mu^{(0)}\bigr)}_{\substack{\triangleq \, r_t, ~ t \ge 0}} \Biggr], \label{Eq24}
\end{align}
and $\mu^{(j)}$, $j=0, 1,  \ldots, N$, are probability measures with respect to the density functions $\rho^{(j)}(t, x)$, $j=0, 1, \ldots, N$, respectively, that satisfy the Liouville equations in \eqref{Eq11} and \eqref{Eq12} starting from an initial density function $\rho_0(x_0)$. 
\end{corollary}

\begin{remark} \label{R8}
The proof is based on the idea of comparing the solutions of the Liouville equations to that of the steady-state solutions of the Fokker-Planck equations as $\epsilon \rightarrow 0$, and thus it is omitted.
\end{remark}

\begin{remark} \label{R9}
Note that a closer look at Equation~\eqref{Eq22} (see also Equations~\eqref{Eq23} and \eqref{Eq24} above) shows that $r_{(\sigma, \epsilon)} > r_t$, $\forall t \ge 0$, with respect to some perturbation $(\sigma, \epsilon)$ and further if $\epsilon_1 \le \epsilon_2$, then $r_{(\sigma, \epsilon_1)} \le r_{(\sigma, \epsilon_2)}$.
\end{remark}

\end{document}